\documentclass[11pt,reqno]{amsart}
\usepackage{indentfirst,enumerate,cite,amssymb,amsfonts,amsmath,amsthm,mathrsfs,dsfont}
\usepackage{color}
\usepackage{multicol}
\usepackage[colorlinks=true,linkcolor=blue,citecolor=blue]{hyperref}
\usepackage{enumerate}
\usepackage{geometry}
\numberwithin{equation}{section}

\newtheorem{theorem}{Theorem}[section]

\newtheorem{lemma}[theorem]{Lemma}
\newtheorem{proposition}[theorem]{Proposition}

\newlength{\dhatheight}

\begin{document}
	
	\title[Solvability of a class of evolution operators]{Solvability of a class of evolution operators on compact Lie groups}
	
	%----------Author 1
	\author[Kirilov]{Alexandre Kirilov}
	\address{
		Departamento de Matem\'atica 
		Universidade Federal do Paran\'a  
		CP 19096, CEP 81530-090, Curitiba 
		Brasil}
	\email{akirilov@ufpr.br}
	
	\thanks{This study was financed in part by the Coordena\c c\~ao de Aperfei\c coamento de Pessoal de N\'ivel Superior - Brasil (CAPES) - Finance Code 001. The first and second authors were supported in part by CNPq - Brasil (grants 316850/2021-7, 301573/2025-5 and 420814/2025-6).}
	
	%----------Author 2
	\author[de Moraes]{Wagner A. A. de Moraes}
	\address{
		Departamento de Matem\'atica  
		Universidade Federal do Paran\'a  
		Caixa Postal 19096\\  CEP 81530-090, Curitiba, Paran\'a 
		Brasil}
	\email{wagnermoraes@ufpr.br}
	
	%----------Author 3
	\author[Tokoro]{Pedro M. Tokoro}
	\address{
		Programa de P\'os-Gradua\c c\~ao em Matem\'atica  
		Universidade Federal do Paran\'a  
		Caixa Postal 19096\\  CEP 81530-090, Curitiba, Paran\'a  
		Brasil}
	\email{pedro.tokoro@ufpr.br}


	
	%\date{today}
	%%% ----------------------------------------------------------------------
	
\begin{abstract}
	This paper provides sufficient conditions for the solvability of a class of first-order evolution operators of Vekua-type on the product of a one-dimensional torus and a compact Lie group. The conditions are expressed in terms of the time-dependent coefficients and the spectral behavior of a normalized left-invariant vector field on the group. The three-sphere case is discussed in detail, leading to more explicit criteria, and the main results are further extended to operators defined on finite products of compact Lie groups.
\end{abstract}
	
	%%% ----------------------------------------------------------------------
	\maketitle
	%%% ----------------------------------------------------------------------
	%\tableofcontents

%==================================================%==================================================	
\section{Introduction}
%==================================================%==================================================	

Let $\mathbb{T}\simeq\mathbb{R}/2\pi\mathbb{Z}$ be the one-dimensional torus and let $G$ be a compact Lie group. This paper is concerned with global solvability on $\mathbb{T}\times G$ for a class of first-order evolution operators of the form
\begin{equation}\label{eq:operatorP}
	Pu=\partial_tu-C(t)Xu-A(t)u-B(t)\overline{u},
\end{equation}
where $A,B,C\in C^\infty(\mathbb{T})$ and $X$ is a normalized left-invariant vector field on $G$. 

The operator $P$ in \eqref{eq:operatorP} belongs to the class of Vekua-type equations, introduced by I. N. Vekua \cite{vekua} in his foundational work on generalized analytic functions. From a conceptual viewpoint, these equations extend the Cauchy--Riemann system by allowing a coupling between $u$ and its complex conjugate, a mechanism that both enriches the analytic structure and is unavoidable in several applications. In the classical literature this framework appears, for instance, in models from the theory of elastic shells and in the study of infinitesimal bendings of surfaces, where conjugation-type terms naturally encode geometric constraints. In a different direction, Vekua-type operators and their modern variants provide an effective language for representing and factorizing solutions of fundamental equations of mathematical physics: in particular, as developed by Kravchenko and collaborators (see \cite{Kravchenko2009} and references therein), suitable Vekua formulations lead to explicit integral representations and reduction procedures for equations such as the Schr\"odinger and Dirac equations. Moreover, the subsequent contributions of Bers \cite{Bers} and many others established a robust theory of generalized analytic functions and pseudoanalytic formalism, with further applications in hydrodynamics and elasticity theory.

Our goal is to identify explicit, verifiable conditions on the coefficients in \eqref{eq:operatorP} that guarantee the following notion of solvability: for every $f\in C^\infty(\mathbb{T}\times G)$ there exists a solution $u\in C^\infty(\mathbb{T}\times G)$ of $Pu=f$. The key difficulty lies in the interaction between the time-dependent coefficients and the representation-theoretic spectrum of $X$, together with the small-divisor phenomena created by the periodicity in $t$ and the presence of the conjugation term.

The present work extends, in two directions, previous solvability results for Vekua-type equations on tori. On $\mathbb{T}$ and $\mathbb{T}^n$, solvability can be studied by partial Fourier series in the space variables, reducing the PDE to families of $2\times2$ ODE systems whose periodic boundary conditions lead to arithmetic constraints (see \cite{BDM,AD,KMT24indag,KMT25ZAA,wagner}). We show that an analogous strategy remains effective on a general compact Lie group $G$, provided one replaces the classical Fourier modes by matrix-valued Fourier coefficients indexed by irreducible unitary representations. This allows us to incorporate the symbol of a left-invariant vector field $X$ into the analysis and to obtain conditions that are stable under the passage from scalar to matrix Fourier coefficients.

Following the approach in \cite{BDM}, we consider operators of the form
\[
Pu=\partial_tu-(p_0+i\lambda q(t))Xu-(s(t)+i\delta q(t))u-\alpha q(t)\overline{u},
\]
where $p_0,\lambda,\delta\in\mathbb{R}$, $\alpha\in\mathbb{C}\setminus\{0\}$, and $q,s\in C^\infty(\mathbb{T};\mathbb{R})$ with $q\ge 0$ and $q\not\equiv 0$. 

Our main theorem gives sufficient conditions for $C^\infty$-solvability in terms of: (i) a non-degeneracy requirement ensuring that the diagonalization of the associated $2\times2$ systems is well-defined; (ii) a global nonresonance condition excluding exact periodic obstructions; and (iii) a quantitative lower bound of Diophantine type controlling the denominators produced by the periodic boundary condition. These conditions depend on the averaged coefficients (integrals over one period) and on the spectral data $\{\mu_m(\xi)\}$ of the symbol $\sigma_X(\xi)$ in each irreducible representation $[\xi]\in\widehat{G}$.

We also treat separately the special case $\lambda=0$, where the relevant spectral parameter becomes independent of the representation and one obtains a simpler set of criteria. As an application, we specialize the general results to $G=\mathbb{S}^3\simeq\mathrm{SU}(2)$, where the Fourier theory is particularly concrete and the spectrum of $X$ can be described explicitly; this yields a transparent formulation of the solvability conditions in terms of the half-integer parameter $\ell$ and the weights $m\in\{-\ell,\dots,\ell\}$. Finally, we extend the solvability theory to the case where $G$ is a finite product of compact Lie groups and the drift term involves a sum of left-invariant vector fields acting on different factors, recovering as a corollary a compact-group analogue of earlier torus results.

Finally, note that when $\alpha=0$ the conjugation term vanishes and the equation is no longer of Vekua type; global solvability and hypoellipticity on compact Lie groups in this setting are treated in \cite{KMR20b,KPM21_jde,KKM24_mana}.

The paper is organized as follows. In Section~\ref{sec:Overview} we recall the basic facts on Fourier analysis and symbols on compact Lie groups and present the partial Fourier series framework on $\mathbb{T}\times G$, including the characterization of smooth functions and distributions in terms of decay/growth of their partial Fourier coefficients. Section~\ref{sec:normalform} performs a reduction to a convenient normal form for the drift term. Section~\ref{sec:solvability} contains the main solvability theorem and its proof. In Section~\ref{sec:lambda0} we discuss the case $\lambda=0$ and provide simplified criteria. Section~\ref{sec:example} treats the example $G=\mathbb{S}^3$, and Section~\ref{sec:product} extends the results to products of compact Lie groups.

	%==================================================
	%==================================================
	\section{Overview of harmonic analysis and symbols on compact Lie groups}\label{sec:Overview}
	%==================================================
	%==================================================
	
	In this section we collect the basic facts from harmonic analysis on compact Lie groups and the global symbolic calculus that will be used throughout the paper. We also introduce the partial Fourier series on $\mathbb{T}\times G$ and recall the characterization of smooth functions and distributions in terms of their partial Fourier coefficients. Background material can be found, for instance, in \cite{livropseudo,KMR20}.
	
	%==================================================
	\subsection{Fourier analysis on compact Lie groups}\label{subsec:fourier_on_G} \
	%==================================================
	
	Let $G$ be a compact Lie group of dimension $d=\dim G$, endowed with the normalized Haar measure $\mathrm{d}x$.
	We denote by $\widehat{G}$ the unitary dual of $G$, i.e., the set of equivalence classes $[\xi]$ of continuous, irreducible, unitary representations $\xi:G\to \mathbb{C}^{d_\xi\times d_\xi}$, where $d_\xi=\dim\xi$. As usual, $\xi$ and $\psi$ are equivalent if there exists a unitary $U$ such that $U\xi(x)=\psi(x)U$ for all $x\in G$.

	For $f\in L^1(G)$, the (group) Fourier transform of $f$ at $\xi$ is the matrix
	\[
	\widehat{f}(\xi)=\int_G f(x)\,\xi(x)^*\,\mathrm{d}x\in\mathbb{C}^{d_\xi\times d_\xi}.
	\]
	
	By the Peter--Weyl theorem, the family
	\[
	\mathcal{B}=\Bigl\{\sqrt{d_\xi}\,\xi_{ij}:\ [\xi]\in\widehat{G},\ 1\le i,j\le d_\xi\Bigr\}
	\]
	is an orthonormal basis of $L^2(G)$ (after fixing one representative in each class). Consequently, every $f\in L^2(G)$ admits the Fourier series expansion
	\[
	f(x)=\sum_{[\xi]\in\widehat{G}} d_\xi\,\mathrm{Tr}\bigl(\xi(x)\widehat{f}(\xi)\bigr),
	\]
	with convergence in $L^2(G)$, and the Plancherel identity holds:
	\[
	\|f\|_{L^2(G)}^2=\sum_{[\xi]\in\widehat{G}} d_\xi\, \|\widehat{f} (\xi)\|_{\mathrm{HS}}^2, \quad \text{ with } \|\widehat{f} (\xi)\|_{\mathrm{HS}}^2=\mathrm{Tr}\bigl(\widehat{f}(\xi)\widehat{f}(\xi)^*\bigr).
	\]
	
	For a distribution $u\in\mathcal{D}'(G)$ we define its Fourier transform by
	\[
	\widehat{u}(\xi)_{ij}=\langle u,\overline{\xi_{ji}}\rangle,\qquad 1\le i,j\le d_\xi,
	\]
	where $\langle\cdot,\cdot\rangle$ denotes distributional duality.
	
	Let $\mathcal{L}_G$ be the Laplace--Beltrami operator on $G$. For each $[\xi]\in\widehat{G}$ the matrix coefficients $\xi_{ij}$ are eigenfunctions of $\mathcal{L}_G$ with the same eigenvalue; we write
	\begin{equation}\label{eq:laplacian_eigen_rewrite}
		-\mathcal{L}_G\,\xi_{ij}(x)=\nu_{[\xi]}\,\xi_{ij}(x),\qquad \nu_{[\xi]}\ge 0.
	\end{equation}
	
	We set the standard weight
	\[
	\langle \xi\rangle := (1+\nu_{[\xi]})^{1/2},
	\]
	which corresponds to the eigenvalues of $(I-\mathcal{L}_G)^{1/2}$.
	
%==================================================
\subsection{Symbols and left-invariant vector fields}\label{subsec:symbols_X} \
%==================================================
	
	Let $\mathfrak{g}$ be the Lie algebra of $G$. For $X\in\mathfrak{g}$ we denote by $L_X$ the associated left-invariant vector field acting on $f\in C^\infty(G)$ by
	\[
	(L_X f)(x):=\frac{\mathrm{d}}{\mathrm{d}t}\bigg|_{t=0} f(x\exp(tX)).
	\]
	
	It commutes with left translations: $L_X\pi_L(y)=\pi_L(y)L_X$ for all $y\in G$, where $(\pi_L(y)f)(x)=f(y^{-1}x)$.
	For simplicity we often write $Xf$ in place of $L_Xf$.
	
	Following the global symbolic calculus on compact groups, for a continuous linear operator
	$P:C^\infty(G)\to C^\infty(G)$ its (matrix) symbol at $(x,\xi)\in G\times \mathrm{Rep}(G)$ is
	\[
	\sigma_P(x,\xi):=\xi(x)^*(P\xi)(x)\in\mathbb{C}^{d_\xi\times d_\xi},
	\]
	where $(P\xi)(x)_{ij}:=P(\xi_{ij})(x)$.
	Then, for $f\in C^\infty(G)$,
	\[
	Pf(x)=\sum_{[\xi]\in\widehat{G}} d_\xi\,\mathrm{Tr}\bigl(\xi(x)\,\sigma_P(x,\xi)\,\widehat{f}(\xi)\bigr).
	\]
	
	If $P$ is left-invariant, then $\sigma_P(x,\xi)$ is independent of $x$ and we simply write $\sigma_P(\xi)$; in this case,
	\[
	\widehat{Pf}(\xi)=\sigma_P(\xi)\,\widehat{f}(\xi),\qquad\forall\,[\xi]\in\widehat{G},
	\]
	and by duality the same identity holds for $f\in\mathcal{D}'(G)$. In particular, \eqref{eq:laplacian_eigen_rewrite} yields
	$\sigma_{\mathcal{L}_G}(\xi)=-\nu_{[\xi]}I_{d_\xi}$.
	
	For each $X\in\mathfrak g$, the operator $iX$ is symmetric on $L^2(G)$.
	Hence, for every class $[\xi]\in\widehat{G}$ we may choose a basis of the representation space so that
	$\sigma_{iX}(\xi)$ is diagonal with real entries $\mu_1(\xi),\dots,\mu_{d_\xi}(\xi)\in\mathbb{R}$, i.e.,
	\[
	\sigma_X(\xi)_{mn}= i\,\mu_m(\xi)\,\delta_{mn}.
	\]
	
	Moreover, the eigenvalues $\{\mu_m(\xi)\}_{m=1}^{d_\xi}$ are invariants of the equivalence class $[\xi]$.
	Since $-(\mathcal{L}_G-X^2)$ is a positive operator and commutes with $X^2$, one obtains the standard estimate
	\begin{equation*}\label{eq:mu_bound_rewrite}
		|\mu_m(\xi)|\le \|X\|\,\langle\xi\rangle,\qquad 1\le m\le d_\xi.
	\end{equation*}
	
	Throughout the paper we assume that the left-invariant vector fields under consideration are normalized, i.e., $\|X\|=1$, where $\|\cdot\|$ is the norm induced by the Killing form.

%==================================================
	\subsection{Partial Fourier series} \label{subsec:partial_fourier} \
%==================================================
	
	Let $\mathbb{T}\simeq\mathbb{R}/2\pi\mathbb{Z}$ %,endowed with the normalized measure  $\mathrm{d}t/(2\pi)$, 
	and consider the product manifold $\mathbb{T}\times G$.
	For $f\in L^1(\mathbb{T}\times G)$ and $\xi\in\mathrm{Rep}(G)$, we define the partial Fourier coefficient with respect to $G$ by
	\[
	\widehat{f}(t,\xi)=\int_G f(t,x)\,\xi(x)^*\,\mathrm{d}x\in\mathbb{C}^{d_\xi\times d_\xi},\qquad t\in\mathbb{T},
	\]
	whose entries are
	\[
	\widehat{f}(t,\xi)_{mn}=\int_G f(t,x)\,\overline{\xi_{nm}(x)}\,\mathrm{d}x,\qquad 1\le m,n\le d_\xi.
	\]
	
	Formally (and in the appropriate distributional sense) one has the partial Fourier series
	\begin{equation*}\label{eq:partial_FS_rewrite}
		f(t,x)=\sum_{[\xi]\in\widehat{G}} d_\xi\,\mathrm{Tr}\bigl(\xi(x)\,\widehat{f}(t,\xi)\bigr).
	\end{equation*}
	
	If $u\in\mathcal{D}'(\mathbb{T}\times G)$, we define its partial Fourier coefficient by requiring that, for each $[\xi]\in\widehat{G}$ and
	$1\le m,n\le d_\xi$, the component $\widehat{u}(\cdot,\xi)_{mn}\in\mathcal{D}'(\mathbb{T})$ satisfies
	\[
	\langle \widehat{u}(\cdot,\xi)_{mn},\psi\rangle=\langle u,\psi\otimes\overline{\xi_{nm}}\rangle,
	\qquad \psi\in C^\infty(\mathbb{T}),
	\]
	where $(\psi\otimes\overline{\xi_{nm}})(t,x):=\psi(t)\overline{\xi_{nm}(x)}$.
	
	The decay/growth properties of these partial Fourier coefficients characterize smoothness and distributional regularity on $\mathbb{T}\times G$ (see \cite{KMR20}).
	
	\begin{theorem}[Characterization of $C^\infty(\mathbb{T}\times G)$]\label{thm:Cinf_char_rewrite}
	A function $f$ belongs to $C^\infty(\mathbb{T}\times G)$ if and only if for every $[\xi]\in\widehat{G}$ and $1\le m,n\le d_\xi$,
	the coefficient $\widehat{f}(\cdot,\xi)_{mn}$ belongs to $C^\infty(\mathbb{T})$ and, for every $\beta\in\mathbb{N}_0$ and every $\ell>0$,
	there exists $C_{\beta,\ell}>0$ such that
	\[
	\sup_{t\in\mathbb{T}}\bigl|\partial_t^\beta \widehat{f}(t,\xi)_{mn}\bigr|\le C_{\beta,\ell}\,\langle\xi\rangle^{-\ell},\qquad\forall\,[\xi]\in\widehat{G},\ 1\le m,n\le d_\xi.
	\]
	\end{theorem}
	
	\begin{theorem}[Characterization of $\mathcal{D}'(\mathbb{T}\times G)$]\label{thm:Dprime_char_rewrite}
	A distribution $u$ belongs to $\mathcal{D}'(\mathbb{T}\times G)$ if and only if there exist $K\in\mathbb{N}$ and $C>0$ such that
	\[
	\bigl|\langle \widehat{u}(\cdot,\xi)_{mn},\varphi\rangle\bigr|
	\le C\,p_K(\varphi)\,\langle\xi\rangle^K,\qquad \forall\,\varphi\in C^\infty(\mathbb{T}),\ \forall\,[\xi]\in\widehat{G},\ 1\le m,n\le d_\xi,
	\]
	where
	\[
	p_K(\varphi):=\sum_{\beta=0}^K \|\partial_t^\beta \varphi\|_{L^\infty(\mathbb{T})}.
	\]
	\end{theorem}
	
	We conclude with a combinatorial identity that will be used to control higher-order derivatives of integrating factors.

	\begin{proposition}[Fa\`a di Bruno formula]\label{prop:faa_di_bruno}
		Let $f\in C^\infty(\mathbb{R})$ and $N\in\mathbb{N}$. Then
		\[
		\frac{\mathrm{d}^N}{\mathrm{d}t^N} e^{f(t)}
		= e^{f(t)}\sum_{\gamma\in\Delta(N)}\frac{N!}{\gamma!}
		\prod_{\ell=1}^N \left( \frac{1}{\ell!}\, \frac{\mathrm{d}^\ell} {\mathrm{d}t^\ell} f(t) \right)^{\gamma_\ell},
		\]
		where
		\[
		\Delta(N)=\left\{\gamma=(\gamma_1,\dots,\gamma_N)\in\mathbb{N}_0^N:\ \sum_{\ell=1}^N \ell\,\gamma_\ell=N\right\},
		\qquad \gamma!:=\gamma_1!\cdots\gamma_N!.
		\]
	\end{proposition}
	
%==================================================
%==================================================
	\section{Reduction to a normal form}\label{sec:normalform}
%==================================================
%==================================================
	
	In this section we reduce a part of our operator to a convenient normal form by an explicit conjugation acting diagonally on the partial Fourier coefficients. We keep the notation and functional framework introduced in Section~\ref{sec:Overview}. In particular, for each class $[\xi]\in\widehat{G}$ we fix a representative such that \(	\sigma_X(\xi) = \operatorname{diag} (i\mu_1(\xi),\dots,i\mu_{d_\xi}(\xi)),\) with \(\mu_j(\xi)\in\mathbb{R}.\)
	
	We start with the evolution operator $L: C^\infty(\mathbb{T} \times G) \to C^\infty(\mathbb{T} \times G)$ defined by
	\[
	Lu=\partial_t u-(p(t)+i\lambda q(t))Xu,
	\]
	where $p,q\in C^\infty(\mathbb{T};\mathbb{R})$, $q\not\equiv 0$, and $\lambda\in\mathbb{R}$. Let
	\[
	p_0=\frac{1}{2\pi}\int_0^{2\pi}p(t)\,\mathrm{d}t,
	\qquad
	\mathcal{P}(t)=\int_0^t p(\tau)\,\mathrm{d}\tau-p_0 t,
	\]
	and consider the operator
	\[
	L_0=\partial_t-(p_0+i\lambda q(t))X.
	\]
	
	The next proposition shows that $L$ and $L_0$ are conjugate by a Fourier multiplier depending only on $\mathcal{P}$; in particular, solvability can be studied for $L_0$ without loss of generality.
	
	\begin{proposition}[Kirilov, de Moraes, and Ruzhansky, 2021]\label{prop:normal_form}
		Consider the operator $\Psi$ defined by
		\[
		\Psi u(t,x)=\sum_{[\xi]\in\widehat{G}} d_\xi \sum_{m,n=1}^{d_\xi}
		e^{-i\mu_m(\xi)\mathcal{P}(t)}\,\widehat{u}(t,\xi)_{mn}\,\xi_{nm}(x).
		\]
		Then:
		\begin{enumerate}
			\item[(a)] $\Psi$ is an automorphism of $C^\infty(\mathbb{T}\times G)$ and $\mathcal{D}'(\mathbb{T}\times G)$;
			\item[(b)] $L_0\circ\Psi=\Psi\circ L$;
			\item[(c)] $L$ is solvable if and only if $L_0$ is solvable.
		\end{enumerate}
	\end{proposition}
	
	\begin{proof}
		The argument is analogous to Propositions~4.6--4.8 in \cite{KMR20} (presented there for $G=\mathbb{S}^3$). The extension to an arbitrary compact Lie group follows from the same partial Fourier series framework and the diagonal form of $\sigma_X(\xi)$ fixed above.
	\end{proof}
	
	In view of Proposition~\ref{prop:normal_form}, it is enough to work with the normal form
	\begin{equation}\label{eq:L_normalform}
		L=\partial_t-(p_0+i\lambda q(t))X,
	\end{equation}
	where $p_0,\lambda\in\mathbb{R}$ and $q\in C^\infty(\mathbb{T};\mathbb{R})$, $q\not\equiv 0$.
	
	Henceforth we assume that $q$ does not change sign on $\mathbb{T}$. In particular,
	\[
	q_0:=\int_0^{2\pi}q(\tau)\,\mathrm{d}\tau\neq 0.
	\]
	
	Without loss of generality we take $q(t)\ge 0$ for all $t\in\mathbb{T}$, so that $q_0>0$.
		
	This choice implies that $Q(t)=\int_0^t q(\tau)\,\mathrm{d}\tau$ is nondecreasing and $\widetilde Q(t)=-\int_t^{2\pi}q(\tau)\,\mathrm{d}\tau$ is nonincreasing, a fact that will be repeatedly used to control the exponential weights arising in the variation-of-constants formulas.
	
%==================================================
%==================================================
	\section{Solvability of a Vekua-type evolution operator}\label{sec:solvability}
%==================================================
%==================================================

	In this section we prove a global $C^\infty$-solvability result for a Vekua-type evolution operator on $\mathbb{T}\times G$.  Consider the operator $P:C^\infty(\mathbb{T}\times G)\to C^\infty(\mathbb{T}\times G)$ defined by
	\begin{equation}\label{eq:P_cv}
		Pu=Lu-(s(t)+i\delta q(t))u-\alpha q(t)\overline{u},
	\end{equation}
	where $L$ is given by \eqref{eq:L_normalform}, $s\in C^\infty(\mathbb{T};\mathbb{R})$, $\delta\in\mathbb{R}$, and $\alpha\in\mathbb{C}\setminus\{0\}$.
	Set
	\[
	s_0=\int_0^{2\pi}s(\tau)\,\mathrm{d}\tau,\qquad
	A_0=s_0+i\delta q_0,\qquad
	B_0=\alpha q_0,\qquad
	C_0=2\pi p_0+i\lambda q_0,
	\]
	and for each $[\xi]\in\widehat{G}$ and $1\le m\le d_\xi$ define
	\[
	\rho_m(\xi)=\sqrt{(\lambda\mu_m(\xi)-i\delta)^2+|\alpha|^2},
	\qquad \text{with }\operatorname{Re}\rho_m(\xi)\ge 0.
	\]

	\begin{theorem}\label{thm:main_varcoef}
		Suppose that the operator $P$ in \eqref{eq:P_cv} satisfies:
		\begin{itemize}
			\item[(I)] $|\alpha|\neq|\delta|$;
			\item[(II)] For all $[\xi]\in\widehat{G}$ and $1\le m\le d_\xi$, there is no integer $k\in\mathbb{Z}$ satisfying
			\[
			\begin{cases}
				\operatorname{Re}\bigl(A_0(2\pi k+\mu_m(\xi)\overline{C_0})\bigr)=0,\\[2mm]
				|2\pi k+\mu_m(\xi)C_0|^2=|A_0|^2-|B_0|^2;
			\end{cases}
			\]
			\item[(III)] There exists $M>0$ such that, for all $[\xi]\in\widehat{G}$ with $\langle\xi\rangle\ge M$,
			\[
			\min\left\{
			\left|e^{-\rho_m(\xi) q_0}-e^{i\mu_m(\xi)p_0 2\pi+s_0}\right|,
			\,
			\left|1-e^{-\rho_m(\xi) q_0+i\mu_m(\xi)p_0 2\pi+s_0}\right|
			\right\}\ge \langle\xi\rangle^{-M},
			\]
			for every $1\le m\le d_\xi$.
		\end{itemize}
		
		Then for every $f\in C^\infty(\mathbb{T}\times G)$ there exists $u\in C^\infty(\mathbb{T}\times G)$ such that $Pu=f$.
	\end{theorem}

\noindent\emph{Proof.}
Let $f\in C^\infty(\mathbb{T}\times G)$ and assume that $u\in\mathcal{D}'(\mathbb{T}\times G)$ satisfies $Pu=f$, where $P$ is given by \eqref{eq:P_cv}. Taking partial Fourier coefficients in the $G$-variable, for each $[\xi]\in\widehat{G}$, we obtain
\begin{equation}\label{eq:Fourier_eq_matrix}
	\partial_t \widehat{u}(t,\xi) - (p_0 + i\lambda q(t)) \sigma_X(\xi)\widehat{u}(t,\xi)
	- (s(t)+i\delta q(t))\widehat{u}(t,\xi)
	- \alpha q(t)\widehat{\overline u}(t,\xi)
	= \widehat{f}(t,\xi).
\end{equation}

Since $\sigma_X(\xi) = \mathrm{diag}(i\mu_1(\xi),\dots,i\mu_{d_\xi}(\xi))$, the $(m,n)$-entry of \eqref{eq:Fourier_eq_matrix} reads, for $1\le m,n\le d_\xi$,
\begin{equation}\label{eq:Fourier_eq_component}
	\partial_t \widehat{u}(t,\xi)_{mn}
	-\Bigl((p_0+i\lambda q(t))\,i\mu_m(\xi)+s(t)+i\delta q(t)\Bigr)\widehat{u}(t,\xi)_{mn}
	-\alpha q(t)\widehat{\overline u}(t,\xi)_{mn}
	=\widehat{f}(t,\xi)_{mn}.
\end{equation}

At this point we couple \eqref{eq:Fourier_eq_component} with the corresponding equation for $\widehat{\overline u}(t,\xi)_{mn}$. We use that
\[
\widehat{\overline u}(t,\xi)=\overline{\widehat{u}(t,\overline\xi)}
\quad\text{and}\quad
\sigma_X(\overline\xi)=\overline{\sigma_X(\xi)},
\]
so that taking the complex conjugate of the equation for $\overline\xi$ yields a second relation with the same $\mu_m(\xi)$. Define
\[
w(t,\xi)_{mn}=
\begin{bmatrix}
	\widehat u(t,\xi)_{mn}\\[1mm]
	\widehat{\overline u}(t,\xi)_{mn}
\end{bmatrix},
\qquad
F(t,\xi)_{mn}=
\begin{bmatrix}
	\widehat f(t,\xi)_{mn}\\[1mm]
	\widehat{\overline f}(t,\xi)_{mn}
\end{bmatrix}.
\]

Then \eqref{eq:Fourier_eq_component} and its conjugate can be written as the $2\times2$ system
\begin{equation*}\label{eq:w_system_detailed}
	w'(t,\xi)_{mn}=M(t,\xi)_m\,w(t,\xi)_{mn}+F(t,\xi)_{mn},
\end{equation*}
where
\[
M(t,\xi)_m=
\begin{bmatrix}
	(p_0+i\lambda q(t))\,i\mu_m(\xi)+s(t)+i\delta q(t) & \alpha q(t)\\
	\overline\alpha q(t) & (p_0-i\lambda q(t))\,i\mu_m(\xi)+s(t)-i\delta q(t)
\end{bmatrix}.
\]

 Let $S(t)=\int_0^t s(\tau)\,\mathrm{d}\tau$ and set
\begin{equation}\label{eq:change_y}
	y(t,\xi)_{mn}=e^{-i\mu_m(\xi)p_0 t-S(t)}\,w(t,\xi)_{mn}.
\end{equation}

A direct differentiation shows that $y$ satisfies
\begin{equation}\label{eq:y_system_detailed}
	y'(t,\xi)_{mn}=q(t)\,\widetilde M(\xi)_m\,y(t,\xi)_{mn}+e^{-i\mu_m(\xi)p_0 t-S(t)}F(t,\xi)_{mn},
\end{equation}
where
\begin{equation*}\label{eq:Mtilde_def}
	\widetilde M(\xi)_m=
	\begin{bmatrix}
		-\lambda \mu_m(\xi)+i\delta & \alpha\\
		\overline\alpha & \lambda \mu_m(\xi)-i\delta
	\end{bmatrix}.
\end{equation*}

The periodicity in $t$ now appears as a twisted boundary condition for $y$. Indeed, since $\widehat u(\cdot,\xi)_{mn}$ is a distribution on $\mathbb{T}$, it is $2\pi$-periodic, hence $w(\cdot,\xi)_{mn}$ is $2\pi$-periodic. Using \eqref{eq:change_y} and $S(2\pi)=s_0$ we obtain
\begin{equation}\label{eq:y_bc_detailed}
	y(0,\xi)_{mn}=e^{i\mu_m(\xi)p_0 2\pi+s_0}\,y(2\pi,\xi)_{mn}.
\end{equation}

The eigenvalues of $\widetilde M(\xi)_m$ are $\pm\rho_m(\xi)$, where
\[
\rho_m(\xi)=\sqrt{(\lambda\mu_m(\xi)-i\delta)^2+|\alpha|^2},\qquad \operatorname{Re}\rho_m(\xi)\ge 0.
\]

Assumption (I) implies $\rho_m(\xi)\neq 0$ for all $([\xi],m)$, hence $\widetilde M(\xi)_m$ is diagonalizable, with eigenvectors given by
\[V^\pm(\xi)_m =\begin{bmatrix}\alpha\\ (\lambda\mu_m(\xi)-i\delta)\pm\rho_m(\xi)\end{bmatrix}.\]

Set
\begin{equation*}
	T(\xi)_m =
	\begin{bmatrix}V^+(\xi)_m & V^-(\xi)_m\end{bmatrix}=
	\begin{bmatrix}
		\alpha & \alpha\\
		(\lambda\mu_m(\xi)-i\delta)+\rho_m(\xi) & (\lambda\mu_m(\xi)-i\delta)-\rho_m(\xi)
	\end{bmatrix}.
\end{equation*}

Then
\begin{equation}\label{eq:diag_relation}
	T(\xi)_m^{-1}\widetilde M(\xi)_mT(\xi)_m=\rho_m(\xi)
	\begin{bmatrix}
		1 & 0\\
		0 & -1
	\end{bmatrix},
\end{equation}
and one checks that
\[
T(\xi)_m^{-1}=\frac{-1}{2\alpha\rho_m(\xi)}
\begin{bmatrix}
	(\lambda\mu_m(\xi)-i\delta)-\rho_m(\xi) & -\alpha\\
	-(\lambda\mu_m(\xi)-i\delta)-\rho_m(\xi) & \alpha
\end{bmatrix}.
\]

Now set
\[
z(t,\xi)_{mn}=T(\xi)_m^{-1}y(t,\xi)_{mn},
\]
and 
\[
G(t,\xi)_{mn}=T(\xi)_m^{-1}F(t,\xi)_{mn}=
\begin{bmatrix}G_1(t,\xi)_{mn}\\ G_2(t,\xi)_{mn}\end{bmatrix}.
\]

Using \eqref{eq:y_system_detailed} and \eqref{eq:diag_relation}, we obtain the decoupled system
\begin{equation}\label{eq:z_system_detailed}
	\begin{cases}
		z_1'(t,\xi)_{mn}=\rho_m(\xi)q(t)\,z_1(t,\xi)_{mn}+e^{-i\mu_m(\xi)p_0 t-S(t)}G_1(t,\xi)_{mn},\\
		z_2'(t,\xi)_{mn}=-\rho_m(\xi)q(t)\,z_2(t,\xi)_{mn}+e^{-i\mu_m(\xi)p_0 t-S(t)}G_2(t,\xi)_{mn}.
	\end{cases}
\end{equation}

Furthermore, since  $T(\xi)_m$ is independent of $t$, the boundary condition \eqref{eq:y_bc_detailed} becomes
\begin{equation}\label{eq:z_bc_detailed}
	z(0,\xi)_{mn}=e^{i\mu_m(\xi)p_0 2\pi+s_0}\,z(2\pi,\xi)_{mn}.
\end{equation}

We now solve \eqref{eq:z_system_detailed} explicitly and determine the constants through \eqref{eq:z_bc_detailed}. Let
\[
Q(t)=\int_0^t q(\tau)\,\mathrm{d}\tau,\qquad
\widetilde Q(t)=-\int_t^{2\pi} q(\tau)\,\mathrm{d}\tau,
\]
so that $Q(0)=0$, $Q(2\pi)=q_0$, $\widetilde Q(0)=-q_0$, $\widetilde Q(2\pi)=0$, and $\widetilde Q(t)\le 0$. Solving by variation of constants,  we have
\begin{align}
	z_1(t,\xi)_{mn}
	&=
	-e^{\rho_m(\xi)\widetilde Q(t)}
	\int_t^{2\pi} e^{-\rho_m(\xi)\widetilde Q(\sigma)}e^{-i\mu_m(\xi)p_0\sigma-S(\sigma)}G_1(\sigma,\xi)_{mn}\,\mathrm{d}\sigma \label{eq:z1_formula_detailed}  \\
	& \hspace{5cm} +K_1(\xi)_{mn}e^{\rho_m(\xi)\widetilde Q(t)}, \nonumber \\
	z_2(t,\xi)_{mn}
	&= 	e^{-\rho_m(\xi)Q(t)} \int_0^{t} e^{\rho_m(\xi)Q(\sigma)}e^{-i\mu_m(\xi)p_0\sigma-S(\sigma)}G_2(\sigma,\xi)_{mn}\,\mathrm{d}\sigma  \label{eq:z2_formula_detailed} \\
	&  \hspace{5cm} +K_2(\xi)_{mn}e^{-\rho_m(\xi)Q(t)}. \nonumber
\end{align}

Imposing \eqref{eq:z_bc_detailed} leads to linear equations for $K_1(\xi)_{mn}$ and $K_2(\xi)_{mn}$, namely
\begin{equation}\label{eq:K1_den_detailed}
	\left(e^{-\rho_m(\xi)q_0}-e^{i\mu_m(\xi)p_0 2\pi+s_0}\right)K_1(\xi)_{mn}
	=
	\int_0^{2\pi} e^{-\rho_m(\xi)(q_0+\widetilde Q(\sigma))}e^{-i\mu_m(\xi)p_0\sigma-S(\sigma)}G_1(\sigma,\xi)_{mn}\,\mathrm{d}\sigma,
\end{equation}
and
\begin{equation}\label{eq:K2_den_detailed}
	\left(1-e^{-\rho_m(\xi)q_0+i\mu_m(\xi)p_0 2\pi+s_0}\right)K_2(\xi)_{mn}
	=
	\int_0^{2\pi} e^{\rho_m(\xi)(Q(\sigma)-q_0)}e^{-i\mu_m(\xi)p_0\sigma-S(\sigma)}G_2(\sigma,\xi)_{mn}\,\mathrm{d}\sigma.
\end{equation}

At this stage, the only potential obstruction arises from the denominators in \eqref{eq:K1_den_detailed}--\eqref{eq:K2_den_detailed}. Condition (II) eliminates exact resonances (i.e., zeros of these denominators), while condition (III) is employed later to manage small denominators as $\langle\xi\rangle\to\infty$. We now clarify the resonance exclusion in a more explicit manner.

Assume first that
\[
e^{-\rho_m(\xi)q_0}-e^{i\mu_m(\xi)p_0 2\pi+s_0}=0.
\]

Since $q\ge 0$ and $q\not\equiv 0$, we have $q_0>0$. Then, there exists an integer $k\in\mathbb{Z}$ such that \(-\rho_m(\xi)q_0 = i\mu_m(\xi)p_0 2\pi+s_0+i2\pi k,\) hence
\begin{equation*}\label{eq:rho_from_resonance_plus}
	\rho_m(\xi)= -\,\frac{s_0+i2\pi(\mu_m(\xi)p_0+k)}{q_0}.
\end{equation*}

Similarly, if
\[
1-e^{-\rho_m(\xi)q_0+i\mu_m(\xi)p_0 2\pi+s_0}=0,
\]
then $-\rho_m(\xi)q_0+i\mu_m(\xi)p_0 2\pi+s_0=i2\pi k$ for some $k\in\mathbb{Z}$, i.e.
\begin{equation*}\label{eq:rho_from_resonance_minus}
	\rho_m(\xi)= \frac{s_0+i2\pi(\mu_m(\xi)p_0-k)}{q_0}.
\end{equation*}

In both cases we can write, after possibly replacing $k$ by $-k$, that
\begin{equation}\label{eq:rho_param}
	\rho_m(\xi)=\pm\frac{s_0+i2\pi(\mu_m(\xi)p_0+k)}{q_0},
\end{equation}
so that
\[\operatorname{Re}\rho_m(\xi)=\pm\frac{s_0}{q_0},\qquad
\operatorname{Im}\rho_m(\xi)=\pm\frac{2\pi(\mu_m(\xi)p_0+k)}{q_0}.
\]

On the other hand, by definition,
\[
(\rho_m(\xi))^2=(\lambda\mu_m(\xi)-i\delta)^2+|\alpha|^2
=\lambda^2\mu_m(\xi)^2-\delta^2+|\alpha|^2-i\,2\lambda\delta\,\mu_m(\xi).
\]

Writing $\rho_m(\xi)=a+ib$ with $a=\operatorname{Re}\rho_m(\xi)$ and $b=\operatorname{Im}\rho_m(\xi)$, we have
\[
(\rho_m(\xi))^2=(a+ib)^2=(a^2-b^2)+i(2ab).
\]

Therefore,
\begin{equation}\label{eq:ab_system}
	\begin{cases}
		2ab=-2\lambda\delta\,\mu_m(\xi),\\[2mm]
		a^2-b^2=\lambda^2\mu_m(\xi)^2-\delta^2+|\alpha|^2.
	\end{cases}
\end{equation}

Substituting the values of $a$ and $b$ from \eqref{eq:rho_param} into
\eqref{eq:ab_system} gives
\begin{equation}\label{eq:system_after_substitution}
	\begin{cases}
		\displaystyle
		2\Bigl(\pm\frac{s_0}{q_0}\Bigr)\Bigl(\pm\frac{2\pi(\mu_m(\xi)p_0+k)}{q_0}\Bigr)
		=-2\lambda\delta\,\mu_m(\xi),\\[3mm]
		\displaystyle
		\Bigl(\pm\frac{s_0}{q_0}\Bigr)^2-\Bigl(\pm\frac{2\pi(\mu_m(\xi)p_0+k)}{q_0}\Bigr)^2
		=\lambda^2\mu_m(\xi)^2-\delta^2+|\alpha|^2.
	\end{cases}
\end{equation}

Since the signs cancel in the first line, \eqref{eq:system_after_substitution} is
equivalent to
\begin{equation}\label{eq:system_simplified}
	\begin{cases}
		\displaystyle
		\frac{s_0\,2\pi(\mu_m(\xi)p_0+k)}{q_0^2}=-\lambda\delta\,\mu_m(\xi),\\[3mm]
		\displaystyle
		\frac{s_0^2-\bigl(2\pi(\mu_m(\xi)p_0+k)\bigr)^2}{q_0^2}
		=\lambda^2\mu_m(\xi)^2-\delta^2+|\alpha|^2.
	\end{cases}
\end{equation}

Now recall $A_0=s_0+i\delta q_0$, $B_0=\alpha q_0$, and $C_0=2\pi p_0+i\lambda q_0$.
A direct expansion gives
\[
A_0\bigl(2\pi k+\mu_m(\xi)\overline{C_0}\bigr)
=(s_0+i\delta q_0)\bigl(2\pi k+2\pi\mu_m(\xi)p_0-i\lambda q_0\mu_m(\xi)\bigr),
\]
and hence
\[
\operatorname{Re}\!\left(A_0(2\pi k+\mu_m(\xi)\overline{C_0})\right)
=s_0\,2\pi(\mu_m(\xi)p_0+k)+\lambda\delta\,\mu_m(\xi)\,q_0^2.
\]

Using the first identity in \eqref{eq:system_simplified}, we obtain
\[
\operatorname{Re}\!\left(A_0(2\pi k+\mu_m(\xi)\overline{C_0})\right)=0.
\]

Likewise,
\[
|2\pi k+\mu_m(\xi)C_0|^2
=\bigl(2\pi k+2\pi\mu_m(\xi)p_0\bigr)^2+\lambda^2\mu_m(\xi)^2q_0^2,
\]
with
\[
|A_0|^2=s_0^2+\delta^2q_0^2, \qquad |B_0|^2=|\alpha|^2q_0^2.
\]

Multiplying the second identity in \eqref{eq:system_simplified} by $q_0^2$ yields
\[
\bigl(2\pi k+2\pi\mu_m(\xi)p_0\bigr)^2+\lambda^2\mu_m(\xi)^2q_0^2
=s_0^2+\delta^2q_0^2-|\alpha|^2q_0^2
=|A_0|^2-|B_0|^2,
\]
that is,
\[
|2\pi k+\mu_m(\xi)C_0|^2=|A_0|^2-|B_0|^2.
\]

We have therefore produced an integer $k\in\mathbb{Z}$ satisfying exactly the system
in (II), contradicting the hypothesis. This proves that neither denominator in
\eqref{eq:K1_den_detailed}--\eqref{eq:K2_den_detailed} can vanish.
Hence, we have
\begin{align*}
K_1(\xi)_{mn} & = \int_0^{2\pi} \frac{e^{-\rho_m(\xi)(q_0 + \widetilde Q(\sigma))} e^{-i\mu_m(\xi) p_0 \sigma - S(\sigma)}}{e^{-\rho_m(\xi) q_0} - e^{i\mu_m(\xi) p_0 2\pi + s_0}} G_1(\sigma,\xi)_{mn} \, \mathrm{d}\sigma,\\[3mm]
K_2(\xi)_{mn} & = \int_0^{2\pi} \frac{e^{\rho_m(\xi)(Q(\sigma) - q_0)} e^{-i\mu_m(\xi) p_0 \sigma - S(\sigma)}}{1 - e^{-\rho_m(\xi) q_0 + i\mu_m(\xi) p_0 2\pi + s_0}} G_2(\sigma,\xi)_{mn} \, \mathrm{d}\sigma.
\end{align*}

	We now reconstruct $\widehat u$ and explain how the smoothness of $u$ follows from estimates on the coefficients. 
	
	Since $w=e^{i\mu_m(\xi)p_0 t+S(t)}\,T(\xi)_m\,z$, we have
	\begin{equation}\label{eq:u_reconstruction_detailed}
		\widehat u(t,\xi)_{mn}
		=
		\alpha\,e^{i\mu_m(\xi)p_0 t+S(t)}\Bigl(z_1(t,\xi)_{mn}+z_2(t,\xi)_{mn}\Bigr).
	\end{equation}
	
	By the characterization of $C^\infty(\mathbb{T}\times G)$ in terms of partial Fourier coefficients (Theorem~\ref{thm:Cinf_char_rewrite}), it is enough to show that for every $\beta\in\mathbb{N}_0$ and every $\ell>0$ there exists $C_{\beta,\ell}>0$ such that
	\begin{equation}\label{eq:goal_decay}
		\sup_{t\in\mathbb{T}}\bigl|\partial_t^\beta \widehat u(t,\xi)_{mn}\bigr|
		\le C_{\beta,\ell}\,\langle\xi\rangle^{-\ell},
		\qquad \forall\,[\xi]\in\widehat G,\ 1\le m,n\le d_\xi.
	\end{equation}
	
	Using Leibniz' rule in \eqref{eq:u_reconstruction_detailed}, it suffices to control derivatives of the exponential factor and of $z_1+z_2$.

	To estimate the exponential factor, we write
	\[
	e^{i\mu_m(\xi)p_0 t+S(t)}=e^{i\mu_m(\xi)p_0 t}\,e^{S(t)},
	\qquad S(t)=\int_0^t s(\tau)\,d\tau.
	\]
	
	For $r\in\mathbb{N}_0$, Leibniz' rule yields
	\begin{equation}\label{eq:der_exp_split_final}
		\partial_t^r\!\left(e^{i\mu_m(\xi)p_0 t+S(t)}\right)
		=\sum_{j=0}^r \binom{r}{j} (i\mu_m(\xi)p_0)^{\,r-j}\,e^{i\mu_m(\xi)p_0 t}\,
		\partial_t^{\,j}\!\left(e^{S(t)}\right).
	\end{equation}
	
	We now use Fa\`a di Bruno (Proposition~\ref{prop:faa_di_bruno}) in the form
	\begin{equation}\label{eq:fdb_eS_final}
		\partial_t^{\,j}\!\left(e^{S(t)}\right)
		=
		e^{S(t)}\sum_{\gamma\in\Delta(j)} \frac{j!}{\gamma!}
		\prod_{\ell=1}^{j}\left(\frac{1}{\ell!}\,\partial_t^{\,\ell}S(t)\right)^{\gamma_\ell},
		\qquad j\ge1,
	\end{equation}
	where $\Delta(j)=\{\gamma\in\mathbb{N}_0^j:\sum_{\ell=1}^j \ell\gamma_\ell=j\}$.
	Since $s\in C^\infty(\mathbb{T})$, every derivative $\partial_t^\ell S(t)$ is bounded on $\mathbb{T}$; hence for each $j$ there exists $C_j>0$ such that
	\begin{equation}\label{eq:eS_bound_final}
		\sup_{t\in\mathbb{T}}\left|\partial_t^{\,j}\!\left(e^{S(t)}\right)\right|\le C_j.
	\end{equation}
	
	Using the normalization bound $|\mu_m(\xi)|\lesssim\langle\xi\rangle$ and combining
	\eqref{eq:der_exp_split_final}--\eqref{eq:eS_bound_final}, we obtain: for each $r\in\mathbb{N}_0$ there exists $C_r>0$ such that
	\begin{equation}\label{eq:exp_factor_growth_final}
		\sup_{t\in\mathbb{T}}
		\left|\partial_t^r e^{i\mu_m(\xi)p_0 t+S(t)}\right|
		\le C_r\,\langle\xi\rangle^{r}.
	\end{equation}
	
	Consequently, to prove \eqref{eq:goal_decay} it is enough to show that for each $\beta\in\mathbb{N}_0$ and each $\ell>0$,
	\begin{equation}\label{eq:goal_decay_z_final}
		\sup_{t\in\mathbb{T}}\bigl|\partial_t^\beta z_j(t,\xi)_{mn}\bigr|
		\le C_{\beta,\ell}\,\langle\xi\rangle^{-\ell},
		\qquad j=1,2.
	\end{equation}
	
	\smallskip
	
	We next obtain the decay properties of $G_1$ and $G_2$ appearing in the decoupled system
	\eqref{eq:z_system_detailed}. Recall that $G=T^{-1}F$, where $F=(\widehat f,\widehat{\overline f})^T$.
	From $f\in C^\infty(\mathbb{T}\times G)$ and Theorem~\ref{thm:Cinf_char_rewrite}, we have: for every
	$\beta\in\mathbb{N}_0$ and every $N>0$ there exists $C_{\beta,N}>0$ such that
	\begin{equation*}\label{eq:f_decay_final}
		\sup_{t\in\mathbb{T}}|\partial_t^\beta \widehat f(t,\xi)_{mn}|\le C_{\beta,N}\langle\xi\rangle^{-N}.
	\end{equation*}
	
	Moreover, as recorded earlier, $|\rho_m(\xi)|\lesssim\langle\xi\rangle$ and $|\rho_m(\xi)|^{-1}\lesssim 1$,
	so the entries of $T(\xi)_m^{-1}$ have at most polynomial growth in $\langle\xi\rangle$. It follows that
	$G_1$ and $G_2$ inherit rapid decay: for each $\beta\in\mathbb{N}_0$ and $N>0$ there exists $C_{\beta,N}>0$
	such that
	\begin{equation}\label{eq:G_decay_final}
		\sup_{t\in\mathbb{T}}|\partial_t^\beta G_j(t,\xi)_{mn}|\le C_{\beta,N}\langle\xi\rangle^{-N},
		\qquad j=1,2.
	\end{equation}
	
	\smallskip
	
	We now estimate the derivatives of the integrating factors in \eqref{eq:z1_formula_detailed}--\eqref{eq:z2_formula_detailed}.
	We treat $e^{\rho_m(\xi)\widetilde Q(t)}$; the bound for $e^{-\rho_m(\xi)Q(t)}$ is analogous. Since $\rho_m(\xi)$
	is independent of $t$, Fa\`a di Bruno applied to $t\mapsto \rho_m(\xi)\widetilde Q(t)$ gives, for $r\ge1$,
	\begin{equation*}\label{eq:fdb_erhoQ_final}
		\partial_t^{\,r}\!\left(e^{\rho_m(\xi)\widetilde Q(t)}\right)
		=
		e^{\rho_m(\xi)\widetilde Q(t)}\sum_{\gamma\in\Delta(r)} \frac{r!}{\gamma!}
		\prod_{\ell=1}^{r}\left(\frac{1}{\ell!}\,\partial_t^{\,\ell}(\rho_m(\xi)\widetilde Q(t))\right)^{\gamma_\ell}.
	\end{equation*}
	
	Since $\partial_t^\ell(\rho_m(\xi)\widetilde Q(t))=\rho_m(\xi)\partial_t^\ell\widetilde Q(t)$, we can factor out
	powers of $\rho_m(\xi)$ and use $|\gamma|:=\sum_{\ell=1}^r\gamma_\ell\le r$ to obtain
	\[
	\left|\partial_t^{\,r}\!\left(e^{\rho_m(\xi)\widetilde Q(t)}\right)\right|
	\le
	|e^{\rho_m(\xi)\widetilde Q(t)}|
	\sum_{\gamma\in\Delta(r)} C\,|\rho_m(\xi)|^{|\gamma|}
	\le
	C_r\,|e^{\rho_m(\xi)\widetilde Q(t)}|\,\langle\xi\rangle^{r}.
	\]
	
	Finally, $\widetilde Q(t)\le0$ and $\operatorname{Re}\rho_m(\xi)\ge0$ imply $|e^{\rho_m(\xi)\widetilde Q(t)}|\le1$, hence
	\begin{equation}\label{eq:exp_rhoQ_growth_final}
		\sup_{t\in\mathbb{T}}\left|\partial_t^{\,r} e^{\rho_m(\xi)\widetilde Q(t)}\right|
		\le C_r\,\langle\xi\rangle^{r},
		\qquad r\in\mathbb{N}_0.
	\end{equation}
	
	\smallskip
	
	We now prove \eqref{eq:goal_decay_z_final} for $z_1$; the estimate for $z_2$ follows by the same arguments with
	$Q$ in place of $\widetilde Q$. Decompose $z_1=z_{1,a}+z_{1,b}$, where
	\begin{align*}
		\begin{cases}
			z_{1,a}(t,\xi)_{mn} & = -e^{\rho_m(\xi)\widetilde Q(t)}
			\int_t^{2\pi} e^{-\rho_m(\xi)\widetilde Q(\sigma)} e^{-i\mu_m(\xi) p_0 \sigma-S(\sigma)}G_1(\sigma,\xi)_{mn}\,d\sigma,
			\\[1mm]
			z_{1,b}(t,\xi)_{mn} & = K_1(\xi)_{mn} e^{\rho_m(\xi) \widetilde  {Q}(t)}.
		\end{cases}
	\end{align*}
	
	Fix $\beta\in\mathbb{N}_0$. Differentiating $z_{1,a}$ under the integral sign, we obtain
	\begin{align}
		\partial_t^\beta z_{1,a}(t,\xi)_{mn}
		= &
		-\partial_t^{\,\beta}\!\left(e^{\rho_m(\xi)\widetilde Q(t)}\right)
		\int_t^{2\pi} H(\sigma,\xi)_{mn}\,d\sigma  \label{eq:diff_z1a_final}.\\
		& +\sum_{r=1}^{\beta}\binom{\beta}{r}\,
		\partial_t^{\,\beta-r}\!\left(e^{\rho_m(\xi)\widetilde Q(t)}\right)\,
		\partial_t^{\,r-1}H(t,\xi)_{mn},\label{eq:diff_z1a_final2} \nonumber
	\end{align}
	where
	\[
	H(\sigma,\xi)_{mn}:=
	e^{-\rho_m(\xi)\widetilde Q(\sigma)}e^{-i\mu_m(\xi)p_0\sigma-S(\sigma)}G_1(\sigma,\xi)_{mn}.
	\]
	
	We first bound the integral term in \eqref{eq:diff_z1a_final}. For $\sigma\in[t,2\pi]$ we have
	$\widetilde Q(\sigma)\le \widetilde Q(t)$, hence $\widetilde Q(t)-\widetilde Q(\sigma)\le0$, and therefore
	\[
	\left|e^{\rho_m(\xi)(\widetilde Q(t)-\widetilde Q(\sigma))}\right|
	\le 1,
	\]
	noticing that $\operatorname{Re}\rho_m(\xi)\geq 0$.
	
	It follows that
	\begin{align*}
		\left|\int_t^{2\pi} H(\sigma,\xi)_{mn}\,d\sigma\right|
		&=
		\left|\int_t^{2\pi} e^{\rho_m(\xi)\widetilde Q(t)}e^{-\rho_m(\xi)\widetilde Q(\sigma)}
		e^{-i\mu_m(\xi)p_0\sigma-S(\sigma)}G_1(\sigma,\xi)_{mn}\,d\sigma\right|\\
		&\le
		\int_t^{2\pi} e^{-S(\sigma)}\,|G_1(\sigma,\xi)_{mn}|\,d\sigma
		\le
		C\,\sup_{\sigma\in\mathbb{T}}|G_1(\sigma,\xi)_{mn}|.
	\end{align*}
	
	Using \eqref{eq:exp_rhoQ_growth_final}, we obtain
	\begin{equation}\label{eq:first_term_z1a_final}
		\sup_{t\in\mathbb{T}}
		\left|
		\partial_t^{\,\beta}\!\left(e^{\rho_m(\xi)\widetilde Q(t)}\right)
		\int_t^{2\pi} H(\sigma,\xi)_{mn}\,d\sigma
		\right|
		\le
		C_\beta\,\langle\xi\rangle^{\beta}\,\sup_{\sigma\in\mathbb{T}}|G_1(\sigma,\xi)_{mn}|.
	\end{equation}
	
	Next we estimate $\partial_t^{\,r-1}H(t,\xi)_{mn}$ in \eqref{eq:diff_z1a_final}. By Leibniz' rule,
	\begin{align}
		\partial_t^{\,r-1}H(t,\xi)_{mn}
		&=
		\sum_{j=0}^{r-1}\binom{r-1}{j}\,
		\partial_t^{\,j}\!\left(e^{-\rho_m(\xi)\widetilde Q(t)}e^{-i\mu_m(\xi)p_0 t-S(t)}\right)\,
		\partial_t^{\,r-1-j}G_1(t,\xi)_{mn}.\label{eq:H_diff_final}
	\end{align}
	
	The derivatives of $e^{-i\mu_m(\xi)p_0 t-S(t)}$ satisfy the same polynomial bounds as in
	\eqref{eq:exp_factor_growth_final}, and the derivatives of $e^{-\rho_m(\xi)\widetilde Q(t)}$ are controlled by
	\eqref{eq:exp_rhoQ_growth_final} (the sign does not change the estimate). Hence there exists $C_j>0$ such that
	\[
	\sup_{t\in\mathbb{T}}
	\left|
	\partial_t^{\,j}\!\left(e^{-\rho_m(\xi)\widetilde Q(t)}e^{-i\mu_m(\xi)p_0 t-S(t)}\right)
	\right|
	\le C_j\,\langle\xi\rangle^{j}.
	\]
	
	Substituting into \eqref{eq:H_diff_final} gives
	\begin{equation}\label{eq:H_diff_bound_final}
		\sup_{t\in\mathbb{T}}|\partial_t^{\,r-1}H(t,\xi)_{mn}|
		\le
		C_r\,\langle\xi\rangle^{r-1}\,
		\sup_{0\le j\le r-1}\ \sup_{t\in\mathbb{T}}|\partial_t^{\,j}G_1(t,\xi)_{mn}|.
	\end{equation}
	
	Combining \eqref{eq:exp_rhoQ_growth_final}--\eqref{eq:first_term_z1a_final},
	and \eqref{eq:H_diff_bound_final}, we conclude that, for every $\beta\in\mathbb{N}_0$,
	\begin{equation*}\label{eq:z1a_beta_bound_final}
		\sup_{t\in\mathbb{T}}|\partial_t^\beta z_{1,a}(t,\xi)_{mn}|
		\le
		C_\beta\,\langle\xi\rangle^{\beta}\,
		\sup_{0\le j\le \beta}\ \sup_{t\in\mathbb{T}}|\partial_t^{\,j}G_1(t,\xi)_{mn}|.
	\end{equation*}
	
	Using the rapid decay \eqref{eq:G_decay_final} and choosing $N$ arbitrarily large, we obtain that
	$\partial_t^\beta z_{1,a}(t,\xi)_{mn}$ is rapidly decaying in $\langle\xi\rangle$, uniformly in $t$.
	
	It remains to estimate $z_{1,b}(t,\xi)_{mn}=K_1(\xi)_{mn}e^{\rho_m(\xi)\widetilde Q(t)}$.
	From \eqref{eq:K1_den_detailed}, the numerator is rapidly decaying by \eqref{eq:G_decay_final} and the boundedness
	of the exponential weights. Condition (II) guarantees the denominator never vanishes; however it may become small for
	large $\langle\xi\rangle$. This is precisely where (III) is used: for $\langle\xi\rangle$ large enough,
	\[
	\left|e^{-\rho_m(\xi)q_0}-e^{i\mu_m(\xi)p_0 2\pi+s_0}\right|^{-1}\le \langle\xi\rangle^{M}.
	\]
	
	Hence $K_1(\xi)_{mn}$ has at most polynomial growth, and choosing the decay order in \eqref{eq:G_decay_final} large
	enough yields, for each $N>0$,
	\begin{equation}\label{eq:K1_decay_final}
		|K_1(\xi)_{mn}|\le C_N\,\langle\xi\rangle^{-N}.
	\end{equation}
	
	Combining \eqref{eq:exp_rhoQ_growth_final} with \eqref{eq:K1_decay_final} and Leibniz' rule shows that, for every
	$\beta\in\mathbb{N}_0$ and every $N>0$,
	\[
	\sup_{t\in\mathbb{T}}|\partial_t^\beta z_{1,b}(t,\xi)_{mn}|
	\le C_{\beta,N}\langle\xi\rangle^{-N}.
	\]
	
	Therefore $z_1=z_{1,a}+z_{1,b}$ satisfies the rapid-decay estimate \eqref{eq:goal_decay_z_final}.
	
	The estimate for $z_2$ is obtained in the same way from \eqref{eq:z2_formula_detailed} and \eqref{eq:K2_den_detailed},
	using the second lower bound in (III). With \eqref{eq:goal_decay_z_final} established for $j=1,2$, the reconstruction
	formula \eqref{eq:u_reconstruction_detailed} together with \eqref{eq:exp_factor_growth_final} yields \eqref{eq:goal_decay}.
	Therefore $\widehat u(\cdot,\xi)_{mn}\in C^\infty(\mathbb{T})$ with rapid decay in $[\xi]$, and
	Theorem~\ref{thm:Cinf_char_rewrite} implies $u\in C^\infty(\mathbb{T}\times G)$ and $Pu=f$. \hfill $\square$

%==================================================
%==================================================
	\section{The case $\lambda=0$}\label{sec:lambda0}
%==================================================
%==================================================
	
	In this section we specialize the previous solvability result to the case $\lambda=0$. In this case the spectral parameter $\rho_m(\xi)$ becomes independent of $[\xi]$, and the solvability conditions admit a more explicit form.
	
	When $\lambda=0$, the operator $P$ reads
	\begin{equation}\label{eq:P_lambda0}
		Pu=\partial_t u-p_0Xu-(s(t)+i\delta q(t))u-\alpha q(t)\overline{u},
	\end{equation}
	with $p_0\in\mathbb{R}$, $q,s\in C^\infty(\mathbb{T};\mathbb{R})$, $q\not\equiv 0$, $q\ge 0$, $\delta\in\mathbb{R}$, and $\alpha\in\mathbb{C}\setminus\{0\}$.
	We keep the notation
	\[
	s_0=\int_0^{2\pi}s(\tau)\,\mathrm{d}\tau,\qquad
	q_0=\int_0^{2\pi}q(\tau)\,\mathrm{d}\tau,\qquad
	A_0=s_0+i\delta q_0,\qquad
	B_0=\alpha q_0,\qquad
	C_0=2\pi p_0.
	\]
	
	\begin{theorem}\label{thm:lambda0_cases}
		Suppose that the operator $P$ in \eqref{eq:P_lambda0} satisfies one of the following conditions:
		\begin{itemize}
			\item[(1)] $|B_0| > |A_0|$;
			\item[(2)] $|B_0| \le |A_0|$, $|\alpha| > |\delta|$ and, for all $[\xi] \in \widehat{G}$ and $1 \le m \le d_\xi$, there is no integer $k \in \mathbb{Z}$ satisfying
			\begin{equation}\label{eq:lambda0_resonance}
				\begin{cases}
					\operatorname{Re}(A_0(k + \mu_m(\xi)p_0)) = 0, \\
					4\pi^2 |k + \mu_m(\xi)p_0|^2 = |A_0|^2 - |B_0|^2;
				\end{cases}
			\end{equation}
			\item[(3)] $|\alpha| < |\delta|$ and $s_0 \neq 0$;
			\item[(4)] $|\alpha| < |\delta|$, $s_0 = 0$, for all $[\xi] \in \widehat{G}$ and $1 \le m \le d_\xi$, there is no $k \in \mathbb{Z}$ satisfying \eqref{eq:lambda0_resonance}, and the following Diophantine condition holds:
			
			\medskip\noindent
			\textup{(DC)} There exists $M > 0$ such that for all $[\xi] \in \widehat{G}$ with $\langle\xi\rangle \ge M$,
			\[
			|2\pi k + \mu_m(\xi)p_0 2\pi - q_0\sqrt{\delta^2 - |\alpha|^2}| \ge \langle\xi\rangle^{-M},
			\qquad 1 \le m \le d_\xi.
			\]
		\end{itemize}
		
		Then, for every $f \in C^\infty(\mathbb{T} \times G)$, there exists $u \in C^\infty(\mathbb{T} \times G)$ such that $Pu = f$.
	\end{theorem}
	\begin{proof}
		Let $f \in C^\infty(\mathbb{T} \times G)$ and suppose that 
		$u \in \mathcal{D}'(\mathbb{T} \times G)$ satisfies $Pu=f$. 
		Proceeding as in the proof of Theorem~\ref{thm:main_varcoef}, 
		we take partial Fourier coefficients with respect to $G$ and 
		reduce the equation to a $2\times2$ system of ODEs for each 
		$[\xi] \in \widehat G$ and $1 \le m \le d_\xi$.
		
		When $\lambda=0$, the quantity
		\[
		\rho_m(\xi)=\sqrt{(-i\delta)^2+|\alpha|^2}
		=\sqrt{|\alpha|^2-\delta^2}
		\]
		is independent of $\xi$, and we simply write $\rho$.
		As in the general case, the solution is obtained by variation of constants,
		and periodicity leads to denominators of the form
		\[
		e^{-\rho q_0}-e^{i\mu_m(\xi)p_0 2\pi+s_0}
		\quad\text{and}\quad
		1-e^{-\rho q_0+i\mu_m(\xi)p_0 2\pi+s_0}.
		\]
		
		Thus, global solvability reduces to controlling these quantities.
		
		We first observe that if $|\alpha| \ne |\delta|$, then $\rho \ne 0$,
		so no degeneracy occurs in the diagonalization procedure.
		Moreover, the resonance system excluded in
		\eqref{eq:lambda0_resonance} corresponds precisely to the
		vanishing of the above denominators, as in the proof of
		Theorem~\ref{thm:main_varcoef}.
		
		If $|\alpha|>|\delta|$, then $\rho\in\mathbb R$.
		In this case $e^{-\rho q_0}$ is a positive real number.
		The identity
		\[
		e^{-\rho q_0}=e^{i\mu_m(\xi)p_0 2\pi+s_0+i2\pi k}
		\]
		would imply simultaneously that the imaginary part vanishes
		and that
		\[
		-\rho q_0 = s_0 + i2\pi(\mu_m(\xi)p_0+k),
		\]
		which yields exactly the system \eqref{eq:lambda0_resonance}.
		Hence, under the hypotheses excluding that system,
		the denominators do not vanish.
		Furthermore, if $|B_0|>|A_0|$, then
		$e^{-\rho q_0}\neq e^{s_0}$,
		so the distance between these exponentials is bounded below by
		a strictly positive constant independent of $[\xi]$.
		The same conclusion follows when $|B_0|\le |A_0|$ but
		\eqref{eq:lambda0_resonance} has no integer solution,
		since the trivial representation would otherwise produce a contradiction.
		In both situations we obtain a uniform bound
		\[
		\left|e^{-\rho q_0}-e^{i\mu_m(\xi)p_0 2\pi+s_0}\right|
		+
		\left|1-e^{-\rho q_0+i\mu_m(\xi)p_0 2\pi+s_0}\right|
		\ge C>0,
		\]
		independent of $[\xi]$.
		
		If instead $|\alpha|<|\delta|$, then $\rho=i\omega$ with
		$\omega=\sqrt{\delta^2-|\alpha|^2}\in\mathbb R$,
		so $|e^{-\rho q_0}|=1$.
		If $s_0\neq0$, then
		\[
		|e^{-\rho q_0}-e^{i\mu_m(\xi)p_0 2\pi+s_0}|
		\ge |1-e^{s_0}|,
		\]
		which is strictly positive.
		Hence the denominators are again uniformly bounded away from zero.
		
		The remaining situation occurs when $|\alpha|<|\delta|$ and $s_0=0$.
		In this case the denominators reduce to
		\[
		|e^{i(\mu_m(\xi)p_0 2\pi-q_0\omega)}-1|,
		\]
		up to harmless sign changes.
		They may approach zero only if the phase
		$\mu_m(\xi)p_0 2\pi-q_0\omega$
		approaches an integer multiple of $2\pi$.
		The Diophantine condition (DC) prevents this phenomenon
		by imposing the lower bound
		\[
		|2\pi k+\mu_m(\xi)p_0 2\pi-q_0\omega|
		\ge \langle\xi\rangle^{-M},
		\]
		which is equivalent, by Lemma~\ref{lem:dc_equiv},
		to a polynomial lower bound on
		$|e^{i(\mu_m(\xi)p_0 2\pi-q_0\omega)}-1|$.
		Consequently, the inverses of the denominators grow at most
		polynomially in $\langle\xi\rangle$.
		
		In all cases, therefore, the constants arising from periodicity
		either remain uniformly bounded or have at most polynomial growth.
		Since the forcing terms are rapidly decreasing in $\xi$,
		the same estimates as in Theorem~\ref{thm:main_varcoef}
		show that the Fourier coefficients of $u$
		decay rapidly.
		Hence $u\in C^\infty(\mathbb{T}\times G)$ and $Pu=f$.
	\end{proof}

	\begin{lemma}\label{lem:dc_equiv}
		The Diophantine condition \textup{(DC)} is equivalent to the following:
		\medskip\noindent
		
		\textup{(DC')} There exists $M > 0$ such that
		\[
		[\xi] \in \widehat{G},\ \langle\xi\rangle \ge M
		\ \Longrightarrow\
		\left|e^{i(\mu_m(\xi)p_0 2\pi - q_0\sqrt{\delta^2 - |\alpha|^2})} - 1\right|
		\ge \langle\xi\rangle^{-M},
		\qquad 1 \le m \le d_\xi.
		\]
	\end{lemma}

	\begin{proof}
		Suppose first that \textup{(DC)} fails. Then, for every $\nu\in\mathbb{N}$ there exist $[\xi_\nu]\in\widehat G$, $m_\nu$ and $k_\nu\in\mathbb{Z}$ with $\langle\xi_\nu\rangle>\nu$ such that
		\[
		|2\pi k_\nu+\mu_{m_\nu}(\xi_\nu)p_0 2\pi-q_0\omega|
		< \langle\xi_\nu\rangle^{-\nu}.
		\]
		
		Using that \(|e^{i\theta}-1|^2=2(1-\cos\theta),\) and the inequality $1-\cos\theta\le |\theta|$ for $|\theta|\le1$, we obtain
		\[
		|e^{i(\mu_{m_\nu}(\xi_\nu)p_0 2\pi-q_0\omega)}-1|^2
		< 2\langle\xi_\nu\rangle^{-\nu},
		\]
		which shows that \textup{(DC')} fails.
		
		\medskip
		
		Conversely, suppose that \textup{(DC')} fails. Then there exist $[\xi_\nu]$ and $m_\nu$ such that
		\[
		|e^{i(\mu_{m_\nu}(\xi_\nu)p_0 2\pi-q_0\omega)}-1|
		< \langle\xi_\nu\rangle^{-\nu}.
		\]
		
		Let $k_\nu$ be the integer closest to
		\[
		\mu_{m_\nu}(\xi_\nu)p_0-\frac{q_0\omega}{2\pi}.
		\]
		
		Since $|e^{i\theta}-1|\to0$ implies that $\theta$ approaches an integer multiple of $2\pi$, we obtain
		\[
		|2\pi k_\nu+\mu_{m_\nu}(\xi_\nu)p_0 2\pi-q_0\omega|
		< C\langle\xi_\nu\rangle^{-\nu/2},
		\]
		for $\nu$ sufficiently large, using the inequality
		\(
		1-\cos\theta\ge \theta^2/4,\) for  \(|\theta|\le2
		\), thus \textup{(DC)} fails.
		
		Therefore, \textup{(DC)} and \textup{(DC')} are equivalent.
	\end{proof}	
	
%==================================================
%==================================================
	\section{Example: the case $G=\mathbb{S}^3$}\label{sec:example}
%==================================================
%==================================================
	
	We illustrate the general conditions of Theorem~\ref{thm:main_varcoef} in the case $G=\mathbb{S}^3$. Let $X$ be a normalized left-invariant vector field on $\mathbb{S}^3$. Using the identification $\mathbb{S}^3\simeq \mathrm{SU}(2)$ and composing with an inner automorphism, we may assume without loss of generality that $iX$ coincides with the standard left-invariant vector field $\partial_0$ in the usual Fourier analysis on $\mathbb{S}^3$. With this choice, the symbol of $X$ is diagonal and given by
	\[
	\sigma_X(\ell)_{mn}= i m\,\delta_{mn},\qquad \ell\in \tfrac12\mathbb{N}_0,
	\]
	where $m,n\in\{-\ell,-\ell+1,\dots,\ell\}$. In particular, the eigenvalues of $\sigma_{iX}(\ell)$ are precisely $m\in\{-\ell,\dots,\ell\}$.
	
	Consider the operator $P:C^\infty(\mathbb{T}\times \mathbb{S}^3)\to C^\infty(\mathbb{T}\times \mathbb{S}^3)$ defined by \medskip
	\begin{equation}\label{eq:P_cvS3}
		Pu=\partial_tu+i(p_0+i\lambda q(t))\partial_0 u-(s(t)+i\delta q(t))u-\alpha q(t)\overline u, \medskip
	\end{equation}
	where $q,s\in C^\infty(\mathbb{T};\mathbb{R})$, $\lambda,\delta\in\mathbb{R}$ and $\alpha\in\mathbb{C}$ satisfy the hypotheses of Theorem~\ref{thm:main_varcoef}. We denote
	\begin{align*}
		s_0 &=\int_0^{2\pi}s(\tau)\,\mathrm{d}\tau,\qquad
		q_0=\int_0^{2\pi}q(\tau)\,\mathrm{d}\tau, \\[3mm]
		A_0 &=s_0+i\delta q_0,\quad B_0=\alpha q_0,\quad C_0=2\pi p_0+i\lambda q_0,
		\end{align*}
	and, for each $\ell\in \tfrac12\mathbb{N}_0$ and $m\in\{-\ell,\dots,\ell\}$, we set
	\[
	\rho_m(\ell)=\sqrt{(\lambda m-i\delta)^2+|\alpha|^2},
	\qquad \text{with }\operatorname{Re}\rho_m(\ell)\ge 0.
	\]
	
	\begin{theorem}\label{thm:S3_general}
		Suppose that $P$ in \eqref{eq:P_cvS3} satisfies:
		\begin{itemize}
			\item[(I)] $|\alpha|\neq |\delta|$;
			\item[(II)] there is no $(m,k)\in\tfrac12\mathbb{Z}\times\mathbb{Z}$ such that
			\[
			\begin{cases}
				\operatorname{Re}\bigl(A_0(2\pi k+m\overline{C_0})\bigr)=0,\\[2mm]
				|2\pi k+m C_0|^2=|A_0|^2-|B_0|^2;
			\end{cases}
			\]
			\item[(III)] there exists $M>0$ such that, for all $\ell\in\tfrac12\mathbb{N}_0$ with $\langle\ell\rangle\ge M$,
			\[
			\min\Bigl\{\bigl|e^{-\rho_m(\ell) q_0}-e^{im p_0 2\pi+s_0}\bigr|,
			\ \bigl|1-e^{-\rho_m(\ell) q_0+im p_0 2\pi+s_0}\bigr|\Bigr\}\ge \langle\ell\rangle^{-M},
			\]
			for every $m\in\{-\ell,\dots,\ell\}$.
		\end{itemize}
		
		Then, for every $f\in C^\infty(\mathbb{T}\times \mathbb{S}^3)$, there exists $u\in C^\infty(\mathbb{T}\times \mathbb{S}^3)$ such that $Pu=f$.
	\end{theorem}

	This is a direct specialization of Theorem~\ref{thm:main_varcoef}, and the only change is that the spectral parameter $\mu_m(\xi)$ becomes the scalar $m$, and the representations are indexed by $\ell\in\tfrac12\mathbb{N}_0$ with $m\in\{-\ell,\dots,\ell\}$. 
	
	\medskip
	We now consider the special case $\lambda=0$. In this regime, \eqref{eq:P_cvS3} becomes
	\begin{equation}\label{eq:P_varcoef2S3}
		Pu=\partial_t u+ip_0\partial_0 u-(s(t)+i\delta q(t))u-\alpha q(t)\overline u,
	\end{equation}
	with $p_0\in\mathbb{R}$, $q,s\in C^\infty(\mathbb{T};\mathbb{R})$, $q\not\equiv 0$, $q\ge 0$, $\delta\in\mathbb{R}$, and $\alpha\in\mathbb{C}\setminus\{0\}$.
	
	\begin{theorem}\label{thm:S3_lambda0}
		Suppose that the operator $P$ in \eqref{eq:P_varcoef2S3} satisfies one of the following conditions:
		\begin{itemize}
			\item[(1)] $|B_0|>|A_0|$;
			\item[(2)] $|B_0|\le |A_0|$, $|\alpha|>|\delta|$, and there is no $(m,k)\in \tfrac12\mathbb{Z}\times\mathbb{Z}$ satisfying
			\begin{equation}\label{eq:cv2_caso2S3}
				\begin{cases}
					\operatorname{Re}\bigl(A_0(k+m p_0)\bigr)=0,\\[3mm]
					4\pi^2|k+m p_0|^2=|A_0|^2-|B_0|^2;
				\end{cases}
			\end{equation}
			\item[(3)] $|\alpha|<|\delta|$ and $s_0\neq 0$;
			\item[(4)] $|\alpha|<|\delta|$, $s_0=0$, there is no $(m,k)\in\tfrac12\mathbb{Z}\times\mathbb{Z}$ satisfying \eqref{eq:cv2_caso2S3}, and the following Diophantine condition holds:
			
			\medskip\noindent
			\textup{(DC$_{\mathbb{S}^3}$)} There exists $M>0$ such that, for all $\ell\in\tfrac12\mathbb{N}_0$ with $\langle\ell\rangle\ge M$, we have
			\[
			|2\pi k + m p_0 2\pi-q_0\sqrt{\delta^2-|\alpha|^2}|\ge \langle\ell\rangle^{-M},
			\qquad \forall\, m\in\{-\ell,\dots,\ell\}.
			\]
		\end{itemize}
		
		Then, for every $f\in C^\infty(\mathbb{T}\times \mathbb{S}^3)$, there exists $u\in C^\infty(\mathbb{T}\times \mathbb{S}^3)$ such that $Pu=f$.
	\end{theorem}
	
	\begin{lemma}\label{lem:dcS3_equiv}
		The Diophantine condition \textup{(DC$_{\mathbb{S}^3}$)} is equivalent to:
		\medskip\noindent
		
		\textup{(DC$'_{\mathbb{S}^3}$)} There exists $M>0$ such that, for all $\ell\in\tfrac12\mathbb{N}_0$ with $\langle\ell\rangle\ge M$, we have
		\[
		\bigl|e^{\,i(mp_0 2\pi-q_0\sqrt{\delta^2-|\alpha|^2})}-1\bigr|\ge \langle\ell\rangle^{-M},
		\qquad \forall\, m\in\{-\ell,\dots,\ell\}.
		\]
	\end{lemma}
	
	The proof follows from the same ideas used in Lemma~\ref{lem:dc_equiv}.
	
	%==================================================
	%==================================================
	\section{Product of groups}\label{sec:product}
	%==================================================
	%==================================================
	
	In this section we extend Theorem~\ref{thm:main_varcoef} to the case where the group variable is a finite product of compact Lie groups. The proof follows the same Fourier reduction as before.
	
	Let $G_1,\dots,G_n$ be compact Lie groups with Lie algebras $\mathfrak{g}_1,\dots,\mathfrak{g}_n$, and let $X_j\in\mathfrak{g}_j$ be normalized left-invariant vector fields, $j=1,\dots,n$. For each $j$ we choose representatives $\xi_j\in\widehat{G}_j$ such that the symbols $\sigma_{X_j}(\xi_j)$ are diagonal:
	\[
	\sigma_{X_j}(\xi_j)_{m_j n_j}= i\,\mu_{m_j}(\xi_j)\,\delta_{m_j n_j},
	\qquad 1\le m_j,n_j\le d_{\xi_j},
	\]
	with $\mu_{m_j}(\xi_j)\in\mathbb{R}$.
	
	Let $G=G_1\times\cdots\times G_n$. We consider the evolution operator
	\[
	Lu=\partial_t u-\sum_{j=1}^n (p_j(t)+i\lambda_j q(t))\,X_j u,
	\]
	where $p_j,q\in C^\infty(\mathbb{T};\mathbb{R})$, $q\not\equiv 0$, $q\ge 0$, and $\boldsymbol{\lambda}=(\lambda_1,\dots,\lambda_n)\in\mathbb{R}^n$.
	Set
	\[
	p_{0j}=\frac{1}{2\pi}\int_0^{2\pi}p_j(t)\,\mathrm{d}t,\quad j=1,\dots,n,
	\qquad
	p_0=(p_{01},\dots,p_{0n})\in\mathbb{R}^n,
	\]
	and reduce as before to
	\[
	L_0=\partial_t-\sum_{j=1}^n (p_{0j}+i\lambda_j q(t))\,X_j.
	\]
	
	Let $P:C^\infty(\mathbb{T}\times G)\to C^\infty(\mathbb{T}\times G)$ be defined by
	\begin{equation}\label{eq:P_cvGn}
		Pu=L_0 u-(s(t)+i\delta q(t))u-\alpha q(t)\overline u,
	\end{equation}
	where $s\in C^\infty(\mathbb{T};\mathbb{R})$, $\delta\in\mathbb{R}$ and $\alpha\in\mathbb{C}\setminus\{0\}$.
	
	We denote
	\[
	s_0=\int_0^{2\pi}s(\tau)\,\mathrm{d}\tau,\quad
	q_0=\int_0^{2\pi}q(\tau)\,\mathrm{d}\tau,
	\quad
	A_0=s_0+i\delta q_0,\quad B_0=\alpha q_0,\quad C_0=2\pi p_0+i\boldsymbol{\lambda}q_0,
	\]
	where $\boldsymbol{\lambda}q_0=(\lambda_1 q_0,\dots,\lambda_n q_0)\in\mathbb{R}^n$.
	
	Given $[\xi]=([\xi_1],\dots,[\xi_n])\in\widehat{G}_1\times\cdots\times\widehat{G}_n\simeq\widehat{G}$ and indices $m=(m_1,\dots,m_n)$, we set
	\[
	\mu_m(\xi)=(\mu_{m_1}(\xi_1),\dots,\mu_{m_n}(\xi_n))\in\mathbb{R}^n,
	\qquad
	\boldsymbol{\lambda}\cdot \mu_m(\xi)=\sum_{j=1}^n \lambda_j \mu_{m_j}(\xi_j).
	\]
	
	For each $([\xi],m)$, define $\rho_m(\xi)$ as the element of
	\[
	\left\{\pm\sqrt{(\boldsymbol{\lambda}\cdot\mu_m(\xi)-i\delta)^2+|\alpha|^2}\right\}
	\]
	with non-negative real part.
	
	\begin{theorem}\label{thm:product_groups}
		Suppose that $P$ in \eqref{eq:P_cvGn} satisfies:
		\begin{itemize}
			\item[(I)] $|\alpha|\neq |\delta|$;
			\item[(II)] for all $[\xi]\in\widehat G$ and all admissible indices $m$, there is no $k\in\mathbb{Z}$ satisfying
			\[
			\begin{cases}
				\operatorname{Re}\bigl(A_0(2\pi k+\mu_m(\xi)\cdot \overline{C_0})\bigr)=0,\\[2mm]
				|2\pi k+\mu_m(\xi)\cdot C_0|^2=|A_0|^2-|B_0|^2;
			\end{cases}
			\]
			\item[(III)] there exists $M>0$ such that, for all $[\xi]\in\widehat G$ with $\langle\xi\rangle\ge M$,
			\[
			\min\Bigl\{\bigl|e^{-\rho_m(\xi) q_0}-e^{i\mu_m(\xi)\cdot p_0 2\pi+s_0}\bigr|,
			\ \bigl|1-e^{-\rho_m(\xi) q_0+i\mu_m(\xi)\cdot p_0 2\pi+s_0}\bigr|\Bigr\}
			\ge \langle\xi\rangle^{-M},
			\]
			for all admissible indices $m$.
		\end{itemize}
		
		Then, for every $f\in C^\infty(\mathbb{T}\times G)$ there exists $u\in C^\infty(\mathbb{T}\times G)$ such that $Pu=f$.
	\end{theorem}

	\bibliographystyle{plain}
	
	\bibliography{bibliografia}

\end{document}